\documentclass[a4paper]{article}
\pdfoutput=1
\usepackage{arxiv}

\usepackage[T1]{fontenc}
\usepackage[ascii]{inputenc}
\usepackage[english]{babel}
\usepackage{tikz-cd}

\usepackage{pdfsync}

\usepackage{cite}
\usepackage{nameref,hyperref,url}

\usepackage{amsmath,amsfonts,amsthm,amssymb}
\usepackage{todonotes}

\usepackage{booktabs}       %
\usepackage{nicefrac}       %
\usepackage{microtype}      %
\usepackage{doi}

\allowdisplaybreaks[2]

\newcommand{\R}{\mathbb{R}}
\newcommand{\N}{\mathbb{N}}
\newcommand{\Rn}{\R^n}
\newcommand{\Sn}{\mathbb{S}^{n-1}}
\renewcommand{\O}{\Omega}

\newcommand{\Od}{\O_{\delta}}
\newcommand{\Wd}{\omega_{\delta}}
\newcommand{\dx}{\,\mathrm{d}x}

\newcommand{\grad}{\mathcal{G}_{\delta}}

\newcommand{\diver}{\mathcal{D}_{\delta}}

\newcommand{\U}{\mathbb{U}}
\newcommand{\Uz}{\U_0}

\newcommand{\Udz}{\U_{\delta,0}}

\newcommand{\Q}{\mathbb{Q}}
\newcommand{\Qd}{\mathbb{Q}_{\delta}}
\newcommand{\Qds}{\Qd^{\text{a}}}

\makeatletter
\newcommand{\thickbar}{\mathpalette\@thickbar}
\newcommand{\@thickbar}[2]{{#1\mkern1.5mu\vbox{
  \sbox\z@{\(#1\mkern-1.5mu#2\mkern-1.5mu\)}%
  \sbox\tw@{\(#1\overline{#2}\)}%
  \dimen@=\dimexpr\ht\tw@-\ht\z@-.8\p@\relax
  \hrule\@height.8\p@ %
  \vskip\dimen@
  \box\z@}\mkern1.5mu}
}
\makeatother

\newcommand{\knl}{{\ddot{\kappa}}}
\newcommand{\qnl}{{\ddot{q}}}
\newcommand{\pnl}{{\ddot{p}}}
\newcommand{\rnl}{{\ddot{r}}}

\newcommand{\Id}{{\hat{I}}}
\newcommand{\Ip}{{\check{I}}}

\newcommand{\Idloc}{\Id_{\text{loc}}}
\newcommand{\Iploc}{\Ip_{\text{loc}}}

\newcommand{\ad}{{\hat{a}}}
\newcommand{\ap}{{\check{a}}}

\newcommand{\kadm}{\mathbb{A}}
\newcommand{\kadmd}{\kadm_{\delta}}

\newcommand{\wsto}{\overset{*}{\rightharpoonup}}
\newcommand{\wto}{\rightharpoonup}

\newtheorem{theorem}{Theorem}[section]

\newtheorem{proposition}[theorem]{Proposition}

\theoremstyle{remark}
\newtheorem{remark}[theorem]{Remark}
\theoremstyle{definition}

\DeclareMathOperator{\graph}{graph}
\DeclareMathOperator{\Div}{div}

\DeclareMathOperator*{\argmin}{arg\,min}


\allowdisplaybreaks
\title{The nonlocal Kelvin principle and the dual approach to nonlocal control in the conduction coefficients}

\author{\href{https://orcid.org/0000-0002-3987-7745}{\includegraphics[scale=0.06]{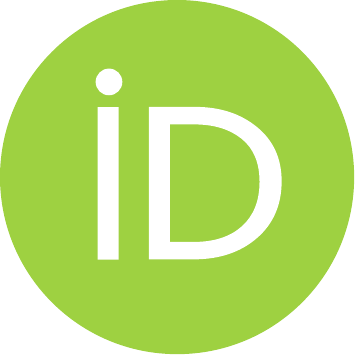}\hspace{1mm}Anton Evgrafov}\\
Department of Mathematical Sciences\\
Aalborg University\\
DK--9210 Aalborg, Denmark\\
\texttt{anev@math.aau.dk}\\
\And
\href{https://orcid.org/0000-0001-5750-1349}{\includegraphics[scale=0.06]{orcid.pdf}\hspace{1mm}Jos\'e C. Bellido}\\
Department of Mathematics\\
University of Castilla-La Mancha\\
13071--Ciudad Real, Spain\\
\texttt{JoseCarlos.Bellido@uclm.es}\\
}

\renewcommand{\shorttitle}{The dual approach to nonlocal control in the conduction coefficients}

\hypersetup{
pdftitle={The nonlocal Kelvin principle and the dual approach to nonlocal control in the conduction coefficients},
pdfauthor={Anton Evgrafov, Jose C. Bellido},
pdfkeywords={Optimal design, nonlocal diffusion, mixed variational principles},
pdfsubject={MSC2010 classification: 49J21, 49J45, 49J35, 80M50},
}

\begin{document}
\maketitle

\begin{abstract} 
We explore the dual approach to nonlocal optimal design, specifically for a classical min-max problem which in this study is associated with a nonlocal scalar diffusion equation.  We reformulate the optimal design problem utilizing a dual variational principle, which is expressed in terms of nonlocal two-point fluxes.  We introduce the proper functional space framework to deal with this formulation, and establish its well-posedness. The key ingredient is the inf-sup (Ladyzhenskaya--Babu{\v{s}}ka--Brezzi) condition, which holds uniformly with respect to small nonlocal horizons.  As a byproduct of this, we are able to prove convergence of nonlocal to local optimal design problems in a straightforward fashion.
\end{abstract}

\keywords{Optimal design \and nonlocal diffusion \and mixed variational principles}

\noindent  \paragraph*{MSC2010 classification:} 49J21, 49J45, 49J35, 80M50

\section{Introduction}\label{sec:intro}

Nolocal models emerge naturally when long-range interactions between points in space and/or in time are considered. Such interactions constitute a shared and important feature in a wide variety of contexts, such as image processing \cite{GiOs08}, pattern formation \cite{Fife03}, population dispersal \cite{CoCoElMa07}, nonlocal diffusion \cite{AnMaRoTo10,BuVal}, nonlocal characterization of Sobolev spaces \cite{bourgain2001another,ponce2004estimate,ponce2004new} and fractional Laplacian \cite{fractional_laplacian} and applications of it, to name just a few. Of particular interest for us is peridynamics, a recent nonlocal paradigm in solid mechanics, see for example \cite{SiLe10, madenci_oterkus}. Our motivation stems from our interest in optimal design of systems governed by nonlocal state laws, which has been previously considered in for example \cite{andres2015nonlocal,evgrafov2019non} in the same context as in the present work, but we also mention applications of nonlocality in shape optimization \cite{BoRiSa}, optimal control \cite{DeGu2014} and inverse problems \cite{DeGu16}.

Optimal design problems have attracted attention of applied mathematicians since the pioneering works on the subject in the sixties and seventies of the last century. We will view optimal design as determining the best way of distributing two different materials, for example conducting or elastic, in desired proportions within a prescribed design domain in order to minimize a given cost functional. Typically, such optimization problems lack optimal solutions, a mathematical pathology which reflects the physical nature of the problem.  Indeed, optimal distributions of given constituents are often microstructures or composites, which require more sophisticated mathematical models in order to be properly described. There have been essentially two distinct ways of dealing with optimal design problems, each leading to different numerical and computational procedures which are nowadays widely utilized in many industries, most notably automotive and aerospace. The first one is based on the homogenization theory for partial differential equations (PDEs) and relaxation of optimal design problems, as the set of admissible material layouts is enlarged to include the composite materials that can obtained as mixtures of the original constituents. The interest reader is referred to the monographs \cite{cherkaev2012variational,allaire2012shape} for a great account of the homogenization method. The second approach to optimal design is based on penalization methods, where the set of admissible distributions is restricted and at the same time, regularization techniques are introduced into the problem. The most successful technique within this class, which is nowadays integrated into many extensively used tools in engineering design,  is the so-called solid isotropic material with penalization (SIMP) model \cite{bendsoe2013topology}.  Interestingly enough, certain regularizing techniques which are utilized within this approach introduce a ``hidden nonlocality'' into the model, as it was pointed out in \cite{sigmund_maute} in the context of SIMP, and subsequently rigorously established in the framework of nonlocal optimal design in conductivity in \cite{evgrafov2019sensitivity}.

Optimal design problems for nonlocal models has been introduced  and treated in \cite{andres2015nonlocal,andres2015type,andres2017convergence,AnMuRo19,AnMuRo21}, and also by the authors in \cite{evgrafov2019non}. These references are pioneering works on the subject of nonlocal optimal design, for which development is still in its infancy. The motivation of considering optimal design under nonlocal laws is twofold. On the one hand, the resulting nonlocal design problems  are generally well posed, and no external artificial regularization is needed.  On the other hand, such problems converge in a certain sense to the corresponding local optimal design problems, or to a regularized or relaxed version of them, as the nonlocality parameter, usually referred to as the {\it horizon}, is driven to zero. Therefore the study of nonlocal optimal design problems is interesting and attractive as approximate optimal solutions to ill-posed local problems can be found as optimal solutions to  ({\it well-posed}) nonlocal problems. Additionally, nonlocal modelling has shown to be very effective in the context of solid mechanics by allowing to effortlessly model certain phenomena, such as fracture and crack propagation, cavitation, and others (see for instance the account \cite{du} and the references therein), which in the context of traditional, PDE-based local mechanical modelling require rather special treatment. %

Dual approach in local optimal design is now a classical chapter of the theory and has resulted in several important theoretical results. The most celebrated example is perhaps associated with utilization of the principle of minimal complementary energy for compliance optimization in elasticity, which allowed to identify and characterize optimal composites in terms of sequential laminates of the constituents materials \cite{allaire_kohn,allaire2012shape}. In this work, we introduce the dual approach to nonlocal optimal design, specifically in the context of the energy minimization problem for a scalar diffusion equation.  This model arises for example when one considers finding an optimal cross-section of an elastic rod in torsion, an optimal thickness of a thin elastic membrane, or most naturally an optimal distribution of electrically or thermally conductive materials.  We reformulate the optimal design problem through a dual variational principle, which can be viewed as a nonlocal version of Kelvin principle.  We introduce the appropriate functional analytic framework to study this princple, and establish the key result for mixed variational statements: the inf-sup (Ladyzhenskaya--Babu{\v{s}}ka--Brezzi, LBB) condition, from which the well-posedness of the variational principle follows in the standard fashion. The interesting part is that the inf-sup condition is uniform with respect to small nonlocal horizons, and allows us to establish convergence of nonlocal problems towards their local limit in a simple manner. From this perspective, this work complements the studies \cite{andres2015nonlocal,andres2015type,andres2017convergence,evgrafov2019non}, where the direct approach to the same optimal design problem was considered.  We believe that the dual approach is not only interesting on its own merits, but could also be the enabling ingredient for establishing novel results within nonlocal optimal design in the future, as has been the case within the field of local optimal design.  For example, extending the results presented in the current work to the nonlinear case of \(p\)-Laplacian appears to be relatively straightforward.

The outline of the paper is as follows. In section 2 primal and dual problems for a model local optimal design problem for distributing conductive materials are introduced. In section 3 we recall a nonlocal version of the primal optimal design problem. Section 4--6 present the main contributions of this work. In section 4 a nonlocal Kelvin principle is introduced,  the appropriate functional space framework to accommodate it is considered, and finally a nonlocal dual version of the optimal design problem is stated and is shown to be well-posed. In sections 5 and 6 we address the nonlocal-to-local (and vice versa) flux recovering as the nonlocal interaction horizon is driven to zero. As a consequence, nonlocal-to-local convergence of the dual problems, and therefore also of the primal ones, is established.

\section{Primal and dual local optimal design problems}\label{sec:local}

We start by recalling a classical problem, which can be viewed as that of 
distributing spatially varying heat conductive
material in a given design domain in order to minimize the weighted average
steady-state temperature, see~\cite{cea1970example}.
For simplicity, we limit ourselves to the situation described by a Dirichlet
boundary value problem with homogeneous boundary conditions.
Indeed, let \(\O\) be a connected bounded open Lipschitz domain in \(\R^n\), \(n \geq 2\).
Throughout the manuscript we will use \(f:\O\to\R\) to denote the volumetric heat source
in the system, and unless specifically stated otherwise we will assume \(f\in L^2(\O)\).
For given constants \(0<\underline{\kappa} < \overline{\kappa}<\infty\) and
\(\gamma \in (\underline{\kappa},\overline{\kappa})\) we define the sets of
admissible conductivites \(\kappa\), temperatures \(u\), and heat fluxes \(q=-\kappa \nabla u\)
by
\begin{equation}
  \begin{aligned}
    \kadm &= \bigg\{\, \kappa \in L^\infty(\O)
    \mid \kappa(x) \in [\underline{\kappa},\overline{\kappa}], \text{\ for almost all \(x\in\O\)},
    \int_{\O} \kappa(x)\dx \leq \gamma|\O|
    \,\bigg\},\\
    \Uz &= H^1_0(\O), \qquad
    \Q = H(\Div,\O;\Rn),\qquad\text{and}\qquad
    \Q(f)= \{\, q \in \Q \mid \Div q=f \,\}.
    \end{aligned}
\end{equation}
We note that here and throughout the manuscrupt we use \(|\cdot|\) loosly to denote the
``size'' of objects, such as Lebesgue measure of sets and Eucledian norms of vectors,
with no ambiguity arising from the context.
Recall that the steady state temperature of the solid for a given conductivity
distribution \(\kappa \in \kadm\) can be determined using the Dirichlet principle,
namely
\begin{equation}\label{eq:Dirloc}
  \begin{aligned}
    u &= \argmin_{v\in \Uz} \Iploc(\kappa;v), &\quad&\text{where}\\
    \Iploc(\kappa;v) &= \frac{1}{2}\int_{\O} \kappa(x)|\nabla v(x)|^2\dx - \ell(v), &\quad&\text{and}\\
    \ell(v) &= \int_{\O} f(x) v(x)\dx.
  \end{aligned}
\end{equation}
The associated optimal design problem, see~\cite{cea1970example}, amounts to
finding a distribution of the conductivities, which maximizes \(\Iploc\).
Consequently, this problem can be stated as the following saddle-point problem:
\begin{equation}\label{eq:compliance_local_primal}
  \begin{aligned}
    p^*=\max_{\kappa \in\kadm} \min_{u\in \Uz} \Iploc(\kappa;u).
  \end{aligned}
\end{equation}
As a direct consequence of the concavity and upper semicontinuity of the map
\(\kappa\mapsto \min_{u\in \Uz} \Iploc(\kappa;u)\) is is not difficult to check that~\eqref{eq:compliance_local_primal}
admits an optimal solution.
In a very real sense this problem is a prototype for a wide class of
\emph{topology optimization} problems with applications in all branches of engineering
sciences, see for example~\cite{bendsoe2013topology,cherkaev2012variational,allaire2012shape}.

We now replace Dirichlet's variational principle (the principle of minimal potential energy) with that of Kelvin (the principle of the minimal complementary energy) for describing the steady state of the heated solid.
This variational principle is formulated in terms of the heat flux
\(q \in \Q(f)\) as the primary unknown:
\begin{equation}\label{eq:Kelloc}
  \begin{aligned}
    &q = \argmin_{p\in\Q(f)} \Idloc(\kappa;p),\quad\text{where}\\
    &\Idloc(\kappa;p) = \frac{1}{2}\int_{\O} \kappa^{-1}(x)|p(x)|^2\dx.
  \end{aligned}
\end{equation}
It is not difficult to check that the optimal values for the
problems~\eqref{eq:Dirloc} and~\eqref{eq:Kelloc} are related through the equality \(\Idloc(\kappa;q)+\Iploc(\kappa;u)=0\).
Therefore, the saddle point problem~\eqref{eq:compliance_local_primal}
can be replaced with a single minimization problem with respect to both the
control coefficient \(\kappa\) and the flux variable \(q\):
\begin{equation}\label{eq:compliance_local_dual}
    d^*=\min_{(\kappa,q)\in \kadm\times\Q(f)} \Idloc(\kappa;q).
\end{equation}
Not surprisingly, in view of its equivalence with~\eqref{eq:compliance_local_primal},
the problem~\eqref{eq:compliance_local_dual} also attains its minimum,
which can also be viewed as a direct consequence of the convexity of the map
\(\R_+\times \Rn\ni(\kappa,q)\mapsto \kappa^{-1}|q|^2\),
see~\cite{bendsoe2013topology,cherkaev2012variational,allaire2012shape}.
For future reference, we state the proof of this fact.
\begin{proposition}\label{prop:dual_loc_exist}
The problem~\eqref{eq:compliance_local_dual} admits an optimal solution.
\end{proposition}
\begin{proof}
The proof is just an application of the direct method of calculus of variations.

\(\kadm\) is the intersection of a shifted closed ball in \(L^\infty(\O)\),
which is weak\(^*\)-sequentially compact owing to Banach--Alaoglu theorem (see for example~\cite[Theorem~3.16]{brezis2010functional}) and a weak\(^*\) closed affine manifold.
Furthermore, the affine manifold \(\Q(f)\) is closed in \(\Q\), and the functional \(\Idloc\) is coercive with respect to the second argument for \(q \in \Q(f)\),
uniformly for \(\kappa \in \kadm\).

Let \((\kappa_k,q_k)\in \kadm\times \Q(f)\) be a minimizing sequence for~\eqref{eq:compliance_local_dual}; without loss of generality we will assume the sequence \(\{ \Idloc(\kappa_k;q_k) \}\) is non-increasing.
In view of the previous arguments the minimizing sequence contains a convergent subsequence, with respect to the weak\(^*\) topology in the first argument and the weak topology in the second.
Without relabelling this subsequence, we assume that 
\((\kappa_k,q_k) \rightharpoonup (\kappa,q)\).
The limit necessarily belongs to \(\kadm\times \Q(f)\), as the feasible set is closed with respect to this type of convergence.
Furthermore, since \(\Omega\) is bounded,  \(\kappa_k \rightharpoonup \kappa\), weakly in \(L^p(\O)\), for any \(1\leq p < \infty\).
Owing to Mazur's lemma (see for example~\cite[Corollary 3.8]{brezis2010functional}), there is a function \(M:\N\to\N\) and a sequence of sets of nonnegative real numbers \(\{\, \lambda(m)_k\colon k=m,\dots,M(m)\,\}\), such that
\(\sum_{k=1}^{M(m)}\lambda(m)_k=1\) and 
\(\lim_{m\to\infty}\|\sum_{k=m}^{M(m)}\lambda(m)_k (\kappa_k,q_k) - (\kappa,q)\|=0\)
in \(L^p(\O)\times \Q\).
In particular, there is a further subsequence of convex combinations of \((\kappa_k,q_k)\), say \(m'\), which converges almost everywhere towards \( (\kappa,q)\) (see for example~\cite[Theorem 4.9]{brezis2010functional}).
Since the integrand \(\kappa^{-1}(x)|q(x)|^2\) is convex and non-negative, we apply
Fatou's lemma (see for example~\cite[Lemma 4.1]{brezis2010functional}) to arrive at the desired conclusion:
\[
\begin{aligned}
\Idloc(\kappa,q)
&\leq \liminf_{m'\to\infty} \Idloc\left(\sum_{k=m'}^{M(m')}\lambda(m')_k \kappa_k,
\sum_{k=m'}^{M(m')}\lambda(m')_k q_k\right)
\leq 
\liminf_{m'\to\infty}\left[\sum_{k=m'}^{M(m')}\lambda(m')_k
\Idloc(\kappa_k;q_k)\right]
\\&\leq
\liminf_{m'\to\infty}
\Idloc(\kappa_{m'};q_{m'}),
\end{aligned}
\]
thereby showing that \((\kappa,q)\) is the optimal solution to~\eqref{eq:compliance_local_dual}.
\end{proof}

\section{Nonlocal Dirichlet principle and the associated optimal design problem}

Recently, nonlocal versions of the problem~\eqref{eq:compliance_local_primal}
have been introduced and actively studied, see~\cite{andres2015type,andres2015nonlocal,andres2017convergence,
evgrafov2019sensitivity,evgrafov2019non}.
Let \(\delta>0\) be a given parameter, which will be referred to as the nonlocal
interaction horizon throughout the manuscript.
Let further \(\Od = \cup_{x\in \O} B(x,\delta)\), where
\(B(x,\delta) = \{\, z \in \Rn : |z-x|<\delta\,\}\).
\(\Od\) is the set of points, which are at most distance \(\delta\) from \(\O\) and consequently
can ``interact'' with points from \(\O\) in the nonlocal model we consider.
Let also \(\Gamma_{\delta} = \Od\setminus\O\) be the ``nonlocal boundary'' of \(\O\).

The strength of nonlocal interaction as a function of the distance will be encoded in a
radial function \(\Wd: \Rn\to \R_+\) with support in \(B(0,\delta)\),
which will be assumed to satisfy the normalization condition
\(\int_{\Rn} |x|^2\Wd^2(x)\dx = K_{2,n}^{-1}\), where
\[K_{p,n}=\frac{1}{|\Sn|}\int_{\Sn}|s\cdot e|^p\,\mathrm{d}s,\]
\(\Sn\) is the \(n-1\)-dimensional unit sphere,
and \(e\in\Sn\) is an arbitrary unit vector~\cite{bourgain2001another}.
For a function \(u:\Rn\to\R\) we define
\begin{equation}\label{eq:defgrad}
  \grad u(x,x') = [u(x)-u(x')]\Wd(x-x'), \quad (x,x')\in \Rn\times\Rn.
\end{equation}
Unless specifically stated otherwise we will extend functions by zero outside of their explicit domain of definition.
With this in mind we will primarily think of \(\grad\) as a linear, possibly unbounded operator \(\grad \in \mathcal{L}(L^2(\O), L^2(\Od\times\Od))\).
For future reference, we summarize the following properties of this operator.
\begin{proposition}\label{prop:grad0}
  The following statements hold.
  \begin{enumerate}
    \item The domain of \(\grad\),
    \(\Udz = D(\grad) = \{\,u \in L^2(\O) \mid \grad u \in L^2(\Od\times\Od)\,\}\)
    is dense in \(L^2(\O)\).
    \item The kernel of \(\grad\), \(\ker\grad\), is trivial.
    \item Poincar{\'e} type inequality holds: \(\|u\|_{L^2(\O)}\leq C\|\grad u\|_{L^2(\Od\times\Od)}\),
    for all \(u\in \Udz\) and \(0<\delta < \overline{\delta}\), with \(C>0\) independent from \(u\) or \(\delta\).
    \item The graph of \(\grad\), \(\graph \grad = \{\,(u,\pnl)\in L^2(\O)\times L^2(\Od\times\Od)
    \mid \pnl=\grad u\,\}\) is closed in \(L^2(\O)\times L^2(\Od\times\Od)\).
    \item \(\Udz\) equipped with the inner product \((u_1,u_2)_{\Udz}=(\grad u_1,\grad u_2)_{L^2(\Od\times\Od)}\)
    is a Hilbert space.
    \item The range of \(\grad\), \(R(\grad) = \{\, \pnl \in L^2(\Od\times\Od) \mid \exists
    u \in L^2(\O): \pnl=\grad u\,\}\) is closed in \(L^2(\Od\times\Od)\).
  \end{enumerate}
\end{proposition}
\begin{proof}
  \begin{enumerate}
  \item It is sufficient to note that \(C^\infty_c(\O)\subset \Udz\),
  since \(\forall \phi \in C^\infty_c(\O)\) we have
  \((x,x')\mapsto [\phi(x)-\phi(x')]|x-x'|^{-1} \in L^\infty(\Od\times\Od)\)
  and \((x,x')\mapsto \Wd(x-x')|x-x'| \in L^2(\Od\times\Od)\).
  \item \(\grad u\equiv 0\) if and only if \(u\) is constant on \(\Od\); additionally,
  \(u\equiv 0\) on \(\Od\setminus\O\).
  \item See for example, \cite[Lemma 4.2 and Remark 4.3]{evgrafov2019non}.
  \item
  Assume that \((u_k,\pnl_k)\in\graph\grad\), \(k=1,2,\dots\), and
  that \(\lim_{k\to\infty} \|u_k-\hat{u}\|_{L^2(\O)} = \lim_{k\to\infty} \|\pnl_k-\hat{\pnl}\|_{L^2(\Od\times\Od)} = 0\).
  Then, up to a subsequence, we have pointwise convergence
  \([u_k(x)-u_k(x')]\Wd(x-x')\to [\hat{u}(x)-\hat{u}(x')]\Wd(x-x')\)
  and
  \(\pnl_k(x,x')\to \hat{\pnl}(x,x')\), for almost all \((x,x')\in\Od^2\).
  Therefore, necessarily we have \((\hat{u},\hat{\pnl})\in\graph\grad\).
  \item
  This follows from points 4 and 3, where the latter shows the equivalence between the norms
  induced by the graph inner product \((u_1,u_2) = (u_1,u_2)_{L^2(\O)} + (\grad u_1,\grad u_2)_{L^2(\Od\times\Od)}\)
  and the inner product \((u_1,u_2)_{\Udz}=(\grad u_1,\grad u_2)_{L^2(\Od\times\Od)}\)
  \item
  Assume now that \((u_k,\pnl_k)\in\graph\grad\), \(k=1,2,\dots\), and
  that \(\lim_{k\to\infty} \|\pnl_k-\hat{\pnl}\|_{L^2(\Od\times\Od)} = 0\).
  In particular, \(\|\grad u_k\|_{L^2(\Od\times\Od)}\) is bounded.
  Therefore, according to point 5, there is \(\hat{u} \in \Udz\), such that
  \(u_k \wto \hat{u}\), weakly in \(\Udz\).
  Since \(\grad: \Udz \to L^2(\Od\times\Od)\) is a bounded linear operator, by the closed graph theorem and 5, we
  have the convergence \(\pnl_k = \grad u_k \wto \grad \hat{u}\), weakly in \(L^2(\Od\times\Od)\).
  As the weak limit is unique, \(\hat{\pnl} = \grad \hat{u}\), and consequently
  \(\hat{\pnl}\in R(\grad)\).\qedhere
  \end{enumerate}
\end{proof}

It is natural to assume that for each pair of interacting points \((x,x') \in \Od\times\Od\),
the strength of their interaction is affected not only by the distance between the points,
but also by the material properties (in our case, conductivity) at these points.
We will denote by \(\knl:\Od\times\Od \to \R\) the precise nature of this dependence.
It is reasonable to assume that \(\knl(x,x')=\knl(x',x)\) encapsulates the averaged properties
\(\kappa(x)\) and \(\kappa(x')\).
In particular, in~\cite{andres2015type,andres2015nonlocal,andres2017convergence}
the arithmetic averaging \(2\knl(x,x')=\kappa(x)+\kappa(x')\) was assumed, and
\cite{evgrafov2019sensitivity,evgrafov2019non} considered the case of the geometric
averaging \(\knl^2(x,x')=\kappa(x)\kappa(x')\).
We will focus on the case of the harmonic averaging of conductivities
\(\knl(x,x') = 2\kappa(x)\kappa(x')[\kappa(x)+\kappa(x')]^{-1}\), which is the
same as assuming the arithmetic averaging for the resistivities \(\kappa^{-1}\).
For now, we can leave out the precise nature of the dependence of \(\knl\) on \(\kappa\),
but we make blanket assumptions \(\knl(x,x')=\knl(x',x)\) and
\(0<\underline{\kappa}\le\knl\le\overline{\kappa}<+\infty\),
for almost all \((x,x')\in\Od\times\Od\).

With these definitions the nonlocal analogue of the Dirichlet principle amounts
to the fact that the steady-state temperature \(u\in\Udz\) can be defined as
the unique minimizer of the assocated energy functional
\begin{equation}\label{eq:def_ap}
  \begin{aligned}
  \Ip(\knl;v)&=\frac{1}{2}\ap(\knl;v,v) - \ell(v), &\qquad&\text{where}\\
  \ap(\knl;u,v)&=\int_{\Od}\int_{\Od} \knl(x,x')\grad u(x,x')\grad v(x,x')\dx\dx',%
  \end{aligned}
\end{equation}
over all \(v \in \Udz\), which exists for any
\(0<\underline{\kappa}\le\knl\le\overline{\kappa}<+\infty\) as a direct consequence
of the Lax--Milgram lemma in view of Proposition~\ref{prop:grad0}.
This minimum satisfies the optimality conditions
\begin{equation}\label{eq:primalproblem}
  \ap(\knl;u,v)=\ell(v), \quad\forall v\in\Udz.
\end{equation}
Quite naturally, the nonlocal analogue of the saddle-point
problem~\eqref{eq:compliance_local_primal} can now be stated as follows:
\begin{equation}\label{eq:compliance_nl_primal}
  p^*_{\delta}=\max_{\kappa\in\kadmd} \min_{u\in \Udz} \Ip(\knl;u),\\
\end{equation}
where, for simplicity, we define 
\(\kadmd =\{\, \kappa \in L^\infty(\Od) \colon
\kappa|_{\O} \in \kadm \land \kappa|_{\Od\setminus\O} \in [\underline{\kappa},\overline{\kappa}]\,\}\).

Now we state an existence result for this problem:
\begin{theorem}\label{thm:primal_summary}
    Suppose that the map \(\kadmd\ni \kappa \mapsto \knl\), where the latter satisfies our blanket assumptions, is continuous and concave, which for example holds when \(\knl(x,x')\) is arithmetic, geometric, or harmonic average of \(\kappa(x)\) and \(\kappa(x')\). Then the problem~\eqref{eq:compliance_nl_primal} admits an optimal solution,
      for each \(\delta >0\).
\end{theorem}
\begin{proof}
  The proof relies heavily on concavity and is given along the lines of the original work in~\cite{cea1970example}, or~\cite[Proposition 4]{evgrafov2019non} for a more recent reference.
\end{proof}

\section{Nonlocal Kelvin principle and the associated optimal design problem}\label{subsec:nloc_mixed}

We now focus on the alternative, dual, approach to~\eqref{eq:compliance_nl_primal}, which avoids the need to deal with saddle-point problems.
We will focus on the case of harmonic averaging of conductivities, but note that this is certainly not a limiting factor in our work.

Let us define a nonlocal linear divergence operator \(\diver : D(\diver)\subset L^2(\Od\times\Od)\to L^2(\O)\)
to be the negative adjoint of the densely defined nonlocal gradient operator \(\grad\).
That is, we define \(\diver\) to satisfy the integration by parts formula:
\begin{equation}\label{eq:diver}
  (\diver \qnl, v)_{L^2(\O)} = -(\qnl,\grad v)_{L^2(\Od\times\Od)},
  \qquad \forall \qnl\in D(\diver)\subset L^2(\Od\times\Od), v\in \Udz.
\end{equation}
We will denote by \(\Qd\) the domain of \(\diver\), that is,
\(\Qd = \{\, \qnl \in L^2(\Od\times\Od) \mid \diver \qnl \in L^2(\O)\,\}\).
This space will be equipped with the graph inner product
\((\qnl,\pnl)_{\Qd} = (\qnl,\pnl)_{L^2(\Od\times\Od)} + (\diver \qnl,\diver \pnl)_{L^2(\O)}\).
We will denote by \(\Qd(f)\) its closed affine subspace
\(\{\, \qnl \in \Qd \mid \diver \qnl = f \,\}\).
The following proposition summarizes the pertinent properties of \(\diver\).
\begin{proposition}\label{prop:diver0}
  The following statements hold.
  \begin{enumerate}
  \item \(\Qd\) is dense in \(L^2(\Od\times\Od)\).
  \item \(\Qd\) is a Hilbert space.
  \item \(\diver: \Qd \to L^2(\O)\) is bounded and surjective.
  \item There is \(C>0\) such that for each \(v \in L^2(\O)\) and each \(0 < \delta < \overline{\delta}\)
  there is \(\qnl_{v,\delta} \in \Qd\) such that \(\diver\qnl_{v,\delta}=v\) and
  \(\|\qnl_{v,\delta}\|_{\Qd}\leq C\|v\|_{L^2(\O)}\).
  \end{enumerate}
\end{proposition}
\begin{proof}
  \begin{enumerate}
    \item Note that for all sufficiently regular \(\qnl\), for example for \(\qnl\in C^{0,1}(\Od\times\Od)\),
    and for all \(v\in\Udz\) we have, owing to Fubini's theorem:
    \begin{equation*}
      \begin{aligned}
      -(\qnl,\grad v)_{L^2(\Od\times\Od)} &= 
      -\lim_{\epsilon\searrow 0} \iint_{O_\epsilon} \frac{\qnl(x,x')}{|x-x'|}[v(x)-v(x')]|x-x'|\Wd(x-x')\dx'\dx
      \\&=
      \lim_{\epsilon\searrow 0} \iint_{O_\epsilon} \bigg[\frac{\qnl(x,x')}{|x-x'|}v(x')|x-x'|\Wd(x-x')
      -
      \frac{\qnl(x,x')}{|x-x'|}v(x)|x-x'|\Wd(x-x')\bigg]\dx'\dx
      \\&=
       \lim_{\epsilon\searrow 0} \iint_{O_\epsilon}v(x)[\qnl(x',x)-\qnl(x,x')]\Wd(x-x')\dx'\dx
      \\&=\int_{\O} v(x)\underbrace{\bigg\{\int_{\Od}[\qnl(x',x)-\qnl(x,x')]\Wd(x-x')\dx'\bigg\}}_{=\diver \qnl(x)}\dx,
      \end{aligned}
    \end{equation*}
    where \(O_\epsilon = \{\, (x,x')\in\Od\times\Od \colon |x-x'|>\epsilon \,\}\).
    Note that the last equality holds because
    \begin{equation*}
      \begin{aligned}
        \|\diver\qnl\|_{L^2(\O)}^2 &\leq
      \int_{\Od}\bigg\{\int_{\Od}[\qnl(x',x)-\qnl(x,x')]\Wd(x-x')\dx'\bigg\}^2\dx
      \\&\leq
      |\Od|\int_{\Od}\int_{\Od}\bigg[\underbrace{\frac{\qnl(x',x)-\qnl(x,x')}{|x-x'|}}_{\leq \sqrt{2}\|\qnl\|_{C^{0,1}(\Od\times\Od)}}|x-x'|\Wd(x-x')\bigg]^2\dx'\dx
      \leq
      2K_{2,n}^{-1}|\Od|^2 \|\qnl\|_{C^{0,1}(\Od\times\Od)}^2,
      \end{aligned}
    \end{equation*}
    and we have utilized H{\"o}lder's inequality.
    \item It is sufficient to note that \(\diver\) is necessarily a closed operator
    as an adjoint operator.
    \item In view of Proposition~\ref{prop:grad0}, closed range theorem applies to
    \(\diver\), and \(R(\diver) = (\ker\grad)^\perp = L^2(\O)\).
    \item For a fixed \(\delta > 0\), the existence of \(\qnl_{v,\delta}\) with the
    claimed properties follows immediately from the bounded inverse theorem applied to
    the bounded bijection \(\diver: (\ker\diver)^{\perp}\subset \Qd \to L^2(\O)\).
    Therefore, it only remains to show that the constant \(C>0\) can be selected independently
    from \(\delta \in (0,\overline{\delta})\).
    Indeed, let us consider an arbitrary \(v \in L^2(\O)\) and \(\delta \in (0,\overline{\delta})\).
    We know that there is unique \(\qnl_{v,\delta} \in (\ker\diver)^{\perp}\),
    such that \(\diver \qnl_{v,\delta} = v\).
    Note that owing to the closed range theorem, we know that
    \(\qnl_{v,\delta} \in R(\grad)=(\ker\diver)^{\perp}\), and therefore there is
    \(u \in \Udz\) such that \(\qnl_{v,\delta} = \grad u\).
    Such a \(u\) necessarily satisfies the variational statement
    \[
    (v,\tilde{v})_{L^2(\O)} = (\diver\grad u,\tilde{v})_{L^2(\O)}
    =-(\grad u,\grad\tilde{v})_{L^2(\Od\times\Od)}
    =-\ap(1;u,\tilde{v}), \qquad \forall \tilde{v}\in \Udz.
    \]
    In particular, taking $\tilde{v}=u$ and applying Cauchy-Bunyakovsky-Schwarz inequality, we obtain the inequality
    \[ \|\grad u\|^2_{L^2(\Od\times\Od)}\le \|v\|_{L^2(\O)}\|u\|_{L^2(\O)},\]
    and by Proposition~\ref{prop:grad0}, point 3, we have that \(\|\qnl_{v,\delta}\|_{L^2(\Od\times\Od)}=\|\grad u\|_{L^2(\Od\times\Od)}
    =\|u\|_{\Udz} \leq \tilde{C}\|v\|_{L^2(\O)}\)
    with \(\tilde{C}\) independent from \(\delta \in (0,\overline{\delta})\) or \(v\).
    Consequently, \(\|\qnl_{v,\delta}\|_{\Qd} \leq
    (\tilde{C}^2+1)^{1/2}\|v\|_{L^2(\O)}\).
    \qedhere
  \end{enumerate}
\end{proof}

Note that the integration by parts formula~\eqref{eq:diver} and the density
of \(\Udz\) in \(L^2(\O)\) imply that the negative adjoint
\(\diver^* \in \mathcal{B}(L^2(\O), \Qd')\)
of \(\diver \in \mathcal{B}(\Qd,L^2(\O))\)\footnote{Here we identify \(L^2(\O)\) with its dual.}
coinsides with \(\grad\) on a dense set \(\Udz\).
For this reason and with a minimal abuse of notation we will also use
\(\grad\) to refer to \(\diver^*\), with no confusion arising from the context.

\begin{proposition}[Poincar{\'e} inequality for \(\grad=\diver^*\)]\label{prop:Poincare_divstar}
  There is \(C>0\) such that for each \(v \in L^2(\O)\) and each \(0 < \delta < \overline{\delta}\)
  we have the inequality \(\|v\|_{L^2(\O)} \leq C \|\grad v\|_{\Qd'}\).
\end{proposition}
\begin{proof}
  One could for example infer this from Proposition~\ref{prop:diver0}, point 4:
  \[
  \|\grad v\|_{\Qd'}
  = \sup_{\qnl\in\Qd\setminus 0}\frac{\langle \grad v,\qnl\rangle}{\|\qnl\|_{\Qd}}
  = \sup_{\qnl\in\Qd\setminus 0}\frac{(v, \diver \qnl)_{L^2(\O)}}{\|\qnl\|_{\Qd}}
  \geq
  \frac{(v, \diver \qnl_{v,\delta})_{L^2(\O)}}{\|\qnl_{v,\delta}\|_{\Qd}}
  \geq C^{-1}\|v\|_{L^2(\O)},
  \]
  where both \(\qnl_{v,\delta} \in \Qd\) and \(C>0\) are as given in Proposition~\ref{prop:diver0}, point 4.
\end{proof}

We now proceed as in Section~\ref{sec:local} and define a nonlocal
flux variable \(\qnl(x,x') = -\knl(x,x')\grad u(x,x')\in L^2(\Od\times\Od)\),
where \(u \in \Udz\) is the nonlocal steady state temperature.
Then the pair \((\qnl,u)\) satisfies the following mixed variational problem:
\begin{equation}\label{eq:mixedproblem}
  \begin{aligned}
    \ad(\knl;\qnl,\pnl) + b(\pnl,u) &= 0, &\qquad& \forall \pnl \in \Qd,\\
    b(\qnl,v) &= -\ell(v), &\qquad& \forall v \in L^2(\O),\footnotemark\\
  \end{aligned}
\end{equation}
\footnotetext{%
Note that a priori this equation holds only \(\forall v \in \Udz\), but
extends to \(L^2(\O)\) by continuity.}
where
\begin{equation}
  \begin{aligned}
    \ad(\knl;\qnl,\pnl) &= \int_{\Od}\int_{\Od} \knl^{-1}(x,x')\qnl(x,x')\pnl(x,x')\dx\dx',
    &\quad&\text{and}\\
    b(\pnl,v) &= -\int_{\O} \diver\pnl(x)v(x)\dx.
  \end{aligned}
\end{equation}
We will now verify that the system~\eqref{eq:mixedproblem} admits a unique solution
for an arbitrary heat source \(f \in L^2(\O)\).

\begin{proposition}\label{prop:cont}
  Bilinear forms \(\ad(\knl;\cdot,\cdot)\) and \(b(\cdot,\cdot)\) are continuous,
  and \(\ad(\knl;\cdot,\cdot)\) is coercive on \(\ker \diver \subset \Qd\),
  with moduli of continuity and coercivity independent from \(\delta\in(0,\overline{\delta})\) or \(\knl\in\kadmd\) (but of course dependent on
  \(\underline{\kappa}\) and \(\overline{\kappa}\)).
\end{proposition}
\begin{proof}
  Trivial.
\end{proof}

\begin{proposition}\label{prop:LBB}
  The bilinear form \(b\) satisfies LBB condition
  \begin{equation}
    \inf_{v \in L^2(\O)}\sup_{\qnl\in \Qd} \frac{b(\qnl,v)}{\|v\|_{L^2(\O)}\|\qnl\|_{\Qd}}
    \geq c >0,
  \end{equation}
  for some constant \(c\) independent from \(\delta \in (0,\overline{\delta})\).
\end{proposition}
\begin{proof}
  We note that
  \[
  \inf_{v \in L^2(\O)}\sup_{\qnl\in \Qd} \frac{b(\qnl,v)}{\|v\|_{L^2(\O)}\|\qnl\|_{\Qd}}
  =
  \inf_{v \in L^2(\O)}\frac{\|\diver^* v\|_{\Qd'}}{\|v\|_{L^2(\O)}}
  \geq C^{-1},
  \]
  owing to Poincar{\'e}-type inequality (see Proposition~\ref{prop:Poincare_divstar}).
\end{proof}

As a consequence of Propositions~\ref{prop:diver0}, \ref{prop:cont}, and~\ref{prop:LBB}
we have the following standard result for mixed variational problems~\cite{boffi2013mixed}.
\begin{theorem}\label{thm:mixed_exist}
  Problem~\eqref{eq:mixedproblem} admits a unique solution \((\qnl,u)\in \Qd\times L^2(\O)\).
  This solution satisfies the stability estimate
  \begin{equation}
    \|\qnl\|_{\Qd} + \|u\|_{L^2(\O)} \le C \|f\|_{L^2(\O)},
  \end{equation}
  for some \(C>0\) independent from \(\qnl\),  \(u\), \(f\), \(\delta \in (0,\overline{\delta})\), or \(\knl \in \kadmd\)
  (but of course dependent on \(\underline{\kappa}\) and \(\overline{\kappa}\)).
\end{theorem}

Theorem~\ref{thm:mixed_exist} and the discussion preceding the derivation of~\eqref{eq:mixedproblem}
establish the direct rigorous correspondence between the unique solutions to~\eqref{eq:primalproblem}
and~\eqref{eq:mixedproblem} for heat sources \(f\in L^2(\O)\).
We now note that~\eqref{eq:mixedproblem} is nothing but the system of optimality
conditions for the constrained optimization problem, which can be viewed as
the nonlocal Kelvin variational principle:
\begin{equation}\label{eq:nlkelvin}
    \min_{\qnl\in\Qd(f)} \Id(\knl;\qnl),
\end{equation}
where \(2\Id(\knl;\qnl) = \ad(\knl;\qnl,\qnl)\).
It is easy to check that the optimal value of the problem above, which is
uniquely attained owing to Theorem~\ref{thm:mixed_exist},
equals to \(-\min_{u\in \Udz} \Ip(\knl;u)\).
Therefore, the problem~\eqref{eq:compliance_nl_primal} can be equivalently stated
in terms of nonlocal fluxes, and we thereby arrive at the nonlocal analogue of~\eqref{eq:compliance_local_dual}:
\begin{equation}\label{eq:compliance_nl_mixed}
  d^*_{\delta}=\min_{(\kappa,\qnl) \in \kadmd\times\Qd(f)} \Id(\knl;\qnl).
\end{equation}
The existence of solutions to~\eqref{eq:compliance_nl_mixed} follows immediately
from that for~\eqref{eq:compliance_nl_primal} and the equivalence between the
two problems, or it can be established independently from the convexity considerations,
precisely as we have done in Proposition~\ref{prop:dual_loc_exist}.

\begin{remark}\label{rem:asym}
  Note that the optimal flux in~\eqref{eq:nlkelvin}, and consequently also
  in~\eqref{eq:compliance_nl_mixed}, is anti-symmetric,
  that is, \(\qnl(x,x')+\qnl(x',x) = 0\),
  for almost all \((x,x')\in \Od\times\Od\).
  Indeed, let us split an arbitrary \(\qnl \in L^2(\Od\times\Od)\) into
  its symmetric and assymetric parts, that is,
  \(\qnl^s(x,x')=(\qnl(x,x')+\qnl(x',x))/2\) and
  \(\qnl^a(x,x')=(\qnl(x,x')-\qnl(x',x))/2\).
  Then, for each \(v \in \Udz\), and by continuity also for each \(v\in L^2(\O)\),
  we have the equality
  \begin{equation}
    b(\qnl,v) = (\qnl^s+\qnl^a,\grad v)_{L^2(\Od\times\Od)}
    =(\qnl^a,\grad v)_{L^2(\Od\times\Od)}=b(\qnl^a,v),
  \end{equation}
  owing to the fact that \(\grad v(x,x')=-\grad v(x',x)\).
  Furthermore,
  \begin{equation}
    2\Id(\knl;\qnl) = \underbrace{\ad(\knl;\qnl^s,\qnl^s)}_{\geq 0}+\ad(\knl;\qnl^a,\qnl^a)
    +2\underbrace{\ad(\knl;\qnl^s,\qnl^a)}_{= 0}
    \geq \ad(\knl;\qnl^a,\qnl^a)=2\Id(\knl;\qnl^a),
  \end{equation}
  where we have utilized the assumed symmetry of \(\knl\).
  From these the claim follows immediately.

  We introduce the following notation for the closed  subspaces of
  symmetric and anti-symmetric fluxes:
  \(L^p(\Od\times\Od)\):
  \[\begin{aligned}
    L^2_s(\Od\times\Od) &= \{\, \qnl \in L^2(\Od\times\Od)
  \mid \qnl(x,x')-\qnl(x',x)=0, \text{a.e.\ in \(\Od\times\Od\)}\,\}\quad\text{and}\\
  L^2_a(\Od\times\Od) &= \{\, \qnl \in L^2(\Od\times\Od)
\mid \qnl(x,x')+\qnl(x',x)=0, \text{a.e.\ in \(\Od\times\Od\)}\,\}.\\
  \end{aligned}\]
  Note that we have the identity
  \(L^2(\Od\times\Od) = L^2_s(\Od\times\Od) \oplus L^2_a(\Od\times\Od)\),
  and in fact the subspaces are orthogonal complements of each other.

  In the minimization problem~\eqref{eq:nlkelvin} and the associated
  design problem~\eqref{eq:compliance_nl_mixed} we can replace \(\Qd\) with
  its closed subspace, which is therefore a Hilbert space with respect to the induced
  norm, \(\Qds = \Qd\cap L^2_a(\Od\times\Od)\) without altering the solution.
  We will also utilize the notation \(\Qds(f) = \{\, \qnl \in \Qds \mid \diver \qnl = f\,\}\).
\end{remark}

\section{Flux recovery operator and an upper estimate for the local problem}\label{sec:gammaconv2}

For each \(\qnl \in L^{2}(\Od\times\Od)\) let us define the flux recovery operator
\begin{equation}\label{eq:Rdef}
  R_\delta \qnl(x) = \int_{\Od} (x-x')\qnl(x,x')\Wd(x-x')\dx'.
\end{equation}

\begin{proposition}\label{prop:claim0}
  \(R_\delta \in \mathcal{B}(L^{2}(\Od\times\Od),L^2(\Od;\Rn))\)
  and \(\Idloc(\kappa;R_\delta \qnl)\leq \Id(\knl;\qnl)\),
  \(\forall \qnl \in L^2_a(\Od\times\Od)\)\footnote{Also \(\forall \qnl \in L^2_s(\Od\times\Od)\)} and
  \(\forall \kappa \in \kadmd\), where \(\knl^{-1}(x,x') = [\kappa^{-1}(x)+\kappa^{-1}(x')]/2\).
\end{proposition}
\begin{proof}
  We begin by applying Cauchy--Bunyakovsky--Schwarz inequality to obtain the estimate
  \[
  |R_\delta \qnl(x)|^2 %
  \leq \int_{\Od} |\qnl(x,x')|^2 \dx' \int_{\Od} |x-x'|^2\Wd^2(x-x')\dx'
  \leq K_{2,n}^{-1}\int_{\Od} |\qnl(x,x')|^2 \dx',
  \]
  where the last inequality is owing to our assumptions on the kernel \(\Wd(\cdot)\).
  It only remains to integrate both sides with respect to \(x\) to arrive at the first claim.
  
  Let us now take an arbitrary \(\psi \in L^2(\O_\delta;\Rn)\), \(\kappa \in \kadmd\),
  and \(\qnl \in L^2_a(\Od\times\Od)\). 
  We apply Cauchy--Bunyakovsky--Schwarz inequality as follows:
  \[\begin{aligned}
  \int_{\O} \kappa^{-1}(x) R_\delta \qnl(x)\cdot \psi(x)\dx
  &=
  \int_{\O}\int_{\Od} \kappa^{-1}(x) \psi(x)\cdot(x-x')\qnl(x,x')\Wd(x-x')\dx'\dx
  \\&\leq
  \left[\int_{\O}\int_{\Od} \kappa^{-1}(x) |\qnl(x,x')|^2 \dx'\dx\right]^{1/2}
  \left[\int_{\O}\int_{\Od} \kappa^{-1}(x)[\psi(x)\cdot(x-x')\Wd(x-x')]^2 \dx'\dx\right]^{1/2}
  \\&\leq
  2
  \Id^{1/2}(\knl; \qnl)
  \Idloc^{1/2}(\kappa; \psi),
  \end{aligned}\]
  where we have used the fact that for all \(x\in \Od\) we have the inequality
  \begin{equation}\label{eq:BBM0}
    \begin{aligned}
   \int_{\Od}[\psi(x)\cdot(x-x')\Wd(x-x')]^2\dx'
   &\le
    \int_{B(0,\delta)} \bigg[\psi(x)\cdot\frac{z}{|z|}\bigg]^2 |z|^2\Wd^2(z)\,\mathrm{d}z
    \\
    &=\int_0^\delta r^{n-1} r^2 w^2_\delta(r) \left(\int_{\Sn} \left[ \psi(x)\cdot s\right]^2\,\mathrm{d}s \right)\, \mathrm{d}r
    \\&\le |\psi(x)|^2  
    \underbrace{\int_{{\Rn}} |z|^2 w^2_\delta(z) \,\mathrm{d}z}_{=K_{2,n}^{-1}} \underbrace{|\Sn|^{-1}\int_{\Sn} \left[e\cdot s\right]^2\,\mathrm{d}s}_{=K_{2,n}}
    = |\psi(x)|^2,
  \end{aligned}
  \end{equation}
 where \(e\in\Sn\) is an arbitrary unit vector.
 The second claim follows by taking \(\psi = R_\delta \qnl\).
\end{proof}

\begin{proposition}\label{prop:claim1_new}
  Assume that the sequence \(\qnl_{\delta_k} \in \Q_{\delta_k}(f)\), \(\delta_k \in (0,\bar{\delta})\),
  \(\lim_{k\to\infty}\delta_k = 0\), is bounded, and
  \(R_{\delta_k}\qnl_{\delta_k} \rightharpoonup \hat{q}\), weakly
  in \(L^2(\O;\Rn)\).
  Then \(\hat{q} \in \Q(f)\).
\end{proposition}
\begin{proof}
  Let \(\psi \in C^\infty_c(\O)\) be arbitrary.
  Utilizing the definition~\eqref{eq:Rdef} and the second order Taylor series expansion of \(\psi\),
  we obtain the estimate
   \begin{equation*}
    \begin{aligned}
      &\bigg| \int_{\Od} \nabla\psi(x)\cdot R_{\delta}\qnl_\delta(x)\dx + \ell(\psi)\bigg|
      =
      \bigg| \int_{\Od} \nabla\psi(x)\cdot R_{\delta}\qnl_\delta(x)\dx
      +  \int_{\O}\diver \qnl_{\delta}(x) \psi(x)\dx \bigg|
      \\
      &=
      \bigg| \int_{\Od} \nabla\psi(x)\cdot R_{\delta}\qnl_\delta(x)\dx
      -  \int_{\Od}\int_{\Od} \grad \psi(x,x') \qnl_\delta(x,x')\dx'\dx \bigg|
    \\&\leq
    \frac{1}{2}\|\nabla^2\psi\|_{L^\infty(\R^n;\R^{n\times n})}
    \int_{\Od}\int_{\Od}|\qnl_\delta(x,x')||x-x'|^2\Wd(x-x')\dx'\dx
    \\&\leq
    \frac{\delta}{2K_{2,n}^{1/2}}\|\nabla^2\psi\|_{L^\infty(\R^n;\R^{n\times n})}\|\qnl_\delta\|_{L^{2}(\Od\times\Od)},
    \end{aligned}
  \end{equation*}
  where we have utilized the fact that \(\Wd \equiv 0\) for \(|x-x'|>\delta\).
\end{proof}

This preliminary work leads to the following one-sided approximation result.

\begin{proposition}\label{thm:nonlocal_approximation_1sided}
  \(d^* \leq \liminf_{\delta\to 0} d^*_{\delta}\).
\end{proposition}
\begin{proof}
  Let \((\kappa_{\delta},\qnl_\delta)\in \kadmd\times \Qd^a(f)\) be a sequence of optimal
  solutions to~\eqref{eq:compliance_nl_mixed}. 
  We note that \(\kadm\) is weak\(^*\) sequentially compact in \(L^\infty(\O)\) and the uniform bound on \(\qnl_{\delta}\) given by Theorem~\ref{thm:mixed_exist}.  Therefore, in view of Propositions~\ref{prop:claim0} and~\ref{prop:claim1_new}, we can assume that
  for some sequence \(\delta_k\searrow 0\) we have \(\kappa_{\delta_k} \wsto \kappa\in \kadm\), weakly\(^*\)
  in \(L^\infty(\O)\) and \(R_{\delta_k} \qnl_{\delta_k} \wto q \in \Q(f)\), weakly in \(L^2(\O)\).
  Appealing to Proposition~\ref{prop:claim0}, convexity of \(\Idloc\), Mazour's and Fatou's lemmata as in the proof of Proposition~\ref{prop:dual_loc_exist} we arrive at the desired claim.
\end{proof}

\section{Adjoint flux recovery operator and a lower estimate for the local problem}

Let us define the operator
\begin{equation}
    R^*_{\delta}q(x,x')
    = \frac{1}{2}[q(x)+q(x')]\cdot (x-x')\Wd(x-x')
\end{equation}
As the notation suggests, \(R^*_{\delta} \in \mathcal{B}(L^2(\Od;\Rn),L^2_a(\Od\times\Od))\)
is the Hilbert space adjoint of \(R_{\delta}\) given by~\eqref{eq:Rdef}, when the latter is restricted to \(L^2_a(\Od\times\Od)\).

\begin{proposition}\label{prop:rstar1}
Assume that \(q\in C^2_c(\Rn;\Rn)\).
Then \(R^*_{\delta}q\in \Qd\) and 
\[\diver R^*_{\delta}q = \text{p.v.}\int_{B(0,\delta)} [q(x+z)\cdot z \Wd^2(z)]\,\mathrm{d}z
\to \Div q,\]
as \(\delta\searrow 0\), pointwise and strongly in \(L^2(\O)\).
\end{proposition}
\begin{proof}
Clearly \(R^*_{\delta}q\in L^2_a(\Od\times\Od)\).
Let us take an arbitrary \(u\in \Udz\). Then
\begin{equation}\label{eq:pvdiv}
\begin{aligned}
(R^*_{\delta}q, \grad u)_{L^2(\Od\times\Od)}
&=
\frac{1}{2}\int_{\Od}\int_{\Od} [u(x)-u(x')][q(x)+q(x')]\cdot (x-x')
\Wd^2(x-x')\dx'\dx
\\&=
\lim_{\epsilon\searrow 0}
\frac{1}{2}\iint\limits_{\Od\times\Od\cap \{\, |x-x'|>\epsilon\,\}}
[u(x)-u(x')][q(x)+q(x')]\cdot (x-x')
\Wd^2(x-x')\dx'\dx
\\&=
-\lim_{\epsilon\searrow 0}
\int_{\O}u(x)\int_{B(0,\delta)\setminus B(0,\epsilon)} [q(x)+q(x+z)]\cdot z\Wd^2(z)\,\mathrm{d}z\dx,
\end{aligned}
\end{equation}
where we have utilized the continuity of Lebesgue's integral and Fubuni's theorem, see for example~\cite[Theorem~4.5]{brezis2010functional}.
We note that for each \(x\in \O\) and each \(\epsilon>0\), we have the identity
\[\int_{B(0,\delta)\setminus B(0,\epsilon)} q(x)\cdot z\Wd^2(z)\,\mathrm{d}z=0,\]
owing to the radial nature of \(\Wd\).
In view of this identity, for each \(x\in \O\) and each \(\epsilon>0\) we can utilize Taylor's theorem to expand \(q(x+z)\) around \(x\) and obtain the following estimate:
\[
\bigg|\int_{B(0,\delta)\setminus B(0,\epsilon)} [q(x+z)\cdot z \Wd^2(z)]\,\mathrm{d}z\bigg|
\leq
\|\nabla q\|_{L^\infty(\Rn,\R^{n\times n})} \int_{B(0,\delta)\setminus B(0,\epsilon)} |z|^2 \Wd^2(z)\,\mathrm{d}z
\leq K_{2,n}^{-1} \|\nabla q\|_{L^\infty(\Rn,\R^{n\times n})}.
\]
Therefore, we can utilize Lebesgue's dominated convergence theorem, see for example~\cite[Theorem~4.2]{brezis2010functional}, in~\eqref{eq:pvdiv} to conclude that
\[
x\mapsto\diver R^*_{\delta}q(x) = \text{p.v.}\int_{B(0,\delta)} q(x+z)\cdot z \Wd^2(z)\,\mathrm{d}z,
\]
is in \(L^2(\O)\), and consequently \(R^*_{\delta}q \in \Qd^a\).
Using one order higher Taylor series expansion of \(q(x+z)\), we can see that
\begin{equation*}\begin{aligned}
\diver R^*_{\delta}q(x) + O(\delta \|\nabla^2 q\|_{L^\infty(\Rn,\R^{n\times n\times n})})
&=
\lim_{\epsilon\searrow 0}
\int_{B(0,\delta)\setminus B(0,\epsilon)} [z^T \nabla q(x)\cdot z]  \Wd^2(z)\,\mathrm{d}z
\\&=
\sum_{i,j=1}^n
\frac{\partial q_i}{\partial x_j}(x)
\int_{B(0,\delta)} [z\cdot e_i][z\cdot e_j]  \Wd^2(z)\,\mathrm{d}z
= \Div q(x),
\end{aligned}\end{equation*}
where \(e_1,\dots,e_n\) is the standard basis in \(\Rn\).
Note that the integrals on the last line evaluate to the Kroneker symbol \(\delta_{ij}\) owing to the radial nature and normalization of \(\Wd\), in the same fashion as in~\eqref{eq:BBM0}.
\end{proof}

\begin{proposition}\label{prop:rstar2}
Assume that \(q\in C^1_c(\Rn;\Rn)\) and \(\kappa \in \kadmd\).
Then
\[
\lim_{\delta\searrow 0}
\Id(\knl;R^*_{\delta}q)= \Idloc(\kappa;q).
\]
\end{proposition}
\begin{proof}
We proceed with the direct computation.
\[\begin{aligned}
\int_{\Od}\int_{\Od} \knl^{-1}(x,x')|R^*_{\delta}q(x,x')|^2 \dx'\dx
&=
\int_{\Od}\kappa^{-1}(x)\int_{\Od} |R^*_{\delta}q(x,x')|^2\dx'\dx
\\&=
\frac{1}{4}
\int_{\Od} \kappa^{-1}(x)\int_{\Od}|q(x)\cdot(x-x')|^2 \Wd^2(x-x')\dx'\dx
\\&+
\frac{1}{4}
\int_{\Od} \kappa^{-1}(x)\int_{\Od}|q(x')\cdot(x-x')|^2 \Wd^2(x-x')\dx'\dx
\\&+
\frac{1}{2}
\int_{\Od} \kappa^{-1}(x)\int_{\Od} [q(x)\cdot(x-x')][q(x')\cdot(x-x')]\Wd^2(x-x')\dx'\dx
\\
&= \frac{1}{4}I_1 + \frac{1}{4}I_2 + \frac{1}{2}I_3.
\end{aligned}
\]
As before, for each \(x\in \O\), we have
\[
\int_{\Od}|q(x)\cdot(x-x')|^2 \Wd^2(x-x')\dx' = |q(x)|^2,
\]
and consequently \(I_1 = 2\Idloc(\kappa;q) + \int_{\Od\setminus \O} \kappa^{-1}(x)\int_{\Od}|q(x')\cdot(x-x')|^2 \Wd^2(x-x')\dx'\dx 
= 2\Idloc(\kappa;q) + O(\delta)\), with proportionality constant depending on \(n\), \(\underline{\kappa}\), \(\|q\|_{L^\infty(\Rn;\R^{n})}\),
but independent from small \(\delta\).
Furthermore, using Taylor series expansion, we can compute
\[
\int_{\Od}|q(x')\cdot(x-x')|^2 \Wd^2(x-x')\dx' = |q(x)|^2 + O(\delta),
\]
with the proportionality constant depending on \(\|\nabla q\|_{L^\infty(\Rn;\R^{n\times n})}\),
and consequently also \(I_2 = 2\Idloc(\kappa;q) + O(\delta)\).
Finally, using the same arguments we have
\[
\int_{\Od}[q(x')\cdot(x-x')][q(x)\cdot(x-x')] \Wd^2(x-x')\dx' = |q(x)|^2 + O(\delta),
\]
and \(I_3 = 2\Idloc(\kappa;q) + O(\delta)\).
In summary, we can conclude that
\[
\lim_{\delta\searrow 0}
\Id(\knl;R^*_{\delta}q)= \Idloc(\kappa;q).\qedhere
\]
\end{proof}

\begin{proposition}\label{thm:nonlocal_approximation_2sided}
Assume that \(\kappa \in \kadmd\) and \(q\in \Q(f)\) are given.
Then there is a sequence \(\pnl_{\delta} \in \Qd(f)\), such that
\(\limsup_{\delta\searrow 0} \Id(\knl;\pnl_{\delta}) \leq
\Idloc(\kappa;q)\).
In particular,
\(\limsup_{\delta\searrow 0} d^*_{\delta} \leq d^*\).
\end{proposition}
\begin{proof}
Let us fix an arbitrary \(\epsilon >0\).
Since \(C^\infty_c(\Rn;\Rn)\) is dense in \(\Q\) (cf.~\cite[Theorem 2.4]{girault_raviart}),
we can select \(q_{\epsilon} \in C^\infty_c(\Rn;\Rn)\) such that
\(\|q-q_\epsilon\|_{\Q}<\epsilon\). In particular,
\(f_\epsilon = \Div q_\epsilon\) satisfies the estimate \(\|f_\epsilon-f\|_{L^2(\O)}<\epsilon\),
and we have the inequality
\begin{equation}\label{eq:idloc_bound}\begin{aligned}
\left|
\Idloc(\kappa;q_\epsilon)-\Idloc(\kappa;q)
\right|
&=
\left|
\frac{1}{2}
\int_{\O} \kappa^{-1}(x)(q_\epsilon(x)-q(x))\cdot(q_\epsilon(x)+q(x))\dx
\right|
\\&\leq
\frac{1}{2}\|\kappa^{-1}\|_{L^\infty(\O)}\|q_\epsilon-q\|_{L^2(\O;\Rn)}
[\|q_\epsilon\|_{L^2(\O;\Rn)}+\|q\|_{L^2(\O;\Rn)}]
\\&\leq \frac{1}{2}\underline{\kappa}^{-1}\epsilon [2\|q\|_{L^2(\O;\Rn)}+\epsilon]=\epsilon\tilde{c},
\end{aligned}\end{equation}
where we have included in \(\tilde{c}>0\) the terms which remain bounded as \(\epsilon\searrow 0\).
Let us put \(\qnl_{\delta,\epsilon} = R_{\delta}^*q_{\epsilon}\),
and \(f_{\delta,\epsilon} = \diver \qnl_{\delta,\epsilon}\).
According to Proposition~\ref{prop:rstar1}, we can select \(\hat{\delta}>0\),
such that for all \(\delta \in (0,\hat{\delta})\) we have the bound
\[\|f_{\delta,\epsilon} - f\|_{L^2(\O)}
\leq \|\diver \qnl_{\delta,\epsilon} - f_{\epsilon}\|_{L^2(\O)}
+ \|f_{\epsilon} - f\|_{L^2(\O)} 
< \|\diver R_{\delta}^*q_{\epsilon} - \Div q_{\epsilon}\|_{L^2(\O)} + \epsilon < 2\epsilon.\]
Let \(\pnl_{\delta}\in \Qd(f)\) and \(\rnl_{\delta,\epsilon} \in \Qd(f_{\delta,\epsilon})\) be the solutions to~\eqref{eq:nlkelvin}
corresponding to the two different volumetric heat sources.
In view of the stability estimate, Theorem~\ref{thm:mixed_exist}, we have
the bounds \(\|\pnl_{\delta} - \rnl_{\delta,\epsilon}\|_{\Qd}
\leq 2\epsilon C\),
\(\|\pnl_{\delta}\|_{\Qd}\leq C\|f\|_{L^2(\O)}\), and
\(\|\rnl_{\delta,\epsilon}\|_{\Qd}\leq C\|f_{\delta,\epsilon}\|_{L^2(\O)}
\leq C(\|f\|_{L^2(\O)}+2\epsilon)\), with constant \(C>0\) independent from \(\delta\), \(\epsilon\), etc., as described in Theorem~\ref{thm:mixed_exist}.
Therefore, similarly to~\eqref{eq:idloc_bound} we have the following estimate:
\[\begin{aligned}
\left|\Id(\knl;\rnl_{\delta,\epsilon})-\Id(\knl;\pnl_{\delta})\right|
&=
\left|
\frac{1}{2}
\int_{\Od}\int_{\Od} \knl^{-1}(x,x')[\rnl_{\delta,\epsilon}(x,x')-\pnl_{\delta}(x,x')]
[\rnl_{\delta,\epsilon}(x,x')+\pnl_{\delta}(x,x')]\dx\dx'
\right|
\\&\leq
\frac{1}{2}\|\knl^{-1}\|_{L^\infty(\Od\times\Od)}\|\rnl_{\delta,\epsilon}-\pnl_{\delta}\|_{L^2(\Od\times\Od)}
[\|\rnl_{\delta,\epsilon}\|_{L^2(\Od\times\Od)}+\|\pnl_{\delta}\|_{L^2(\Od\times\Od)}]
\\&\leq \underline{\kappa}^{-1}\epsilon C^2[2 \|f\|_{L^2(\O)}+2\epsilon]
= \epsilon\tilde{C},
\end{aligned}\]
where \(\tilde{C}>0\) remains bounded as \(\delta\searrow 0\) and \(\epsilon \searrow 0\).
Consequently, we can write
\[
\limsup_{\delta \searrow 0}\Id(\knl;\pnl_{\delta}) \leq
\limsup_{\delta \searrow 0}
\Id(\knl;\rnl_{\delta,\epsilon}) + \epsilon\tilde{C}
\leq
\limsup_{\delta \searrow 0}
\Id(\knl;\qnl_{\delta,\epsilon}) + \epsilon\tilde{C}
=\Idloc(\kappa;q_\epsilon) +\epsilon \tilde{C}
\leq
\Idloc(\kappa;q) + \epsilon(\tilde{c}+\tilde{C}), 
\]
where we have utilized the optimality of \(\rnl_{\delta,\epsilon}\) as well
as Proposition~\ref{prop:rstar2} and the estimate~\eqref{eq:idloc_bound}.
Since \(\epsilon>0\) is arbitrary, the proof is concluded.
\end{proof}

Now, as a direct consequence of Propositions \ref{thm:nonlocal_approximation_1sided} and \ref{thm:nonlocal_approximation_2sided} we obtain the following nonlocal-to-local approximation result. 

\begin{theorem}\label{thm:primal_summary1}
    Under the previous assumptions, 
    \[d^*=\lim_{\delta\searrow 0} d^*_{\delta},\]
    and consequently 
      \begin{equation}\label{eq:liminf000}
      p^*
      =
      \lim_{\delta\searrow 0}
      p^*_{\delta},
      \end{equation}
      that is, we have convergence of the optimal values of~\eqref{eq:compliance_nl_primal} towards that of~\eqref{eq:compliance_local_primal}
      when the nonlocal interaction horizon converges to \(0\).
      Additionally, let \(\{(\kappa_\delta,\qnl_\delta)\}\) be a sequence of optimal solutions to~\eqref{eq:compliance_nl_mixed}.
      Then, the sequence \(\{(\kappa_\delta,R_\delta \qnl_\delta)\}\) is sequentially compact, with respect to weak\(^*\) topology of \(L^\infty(\O)\)
      and  weak topology of \(L^2(\O;\Rn)\). 
      Each limit point of this sequence is an optimal solution to~\eqref{eq:compliance_local_dual}.
      
      In fact, the sequence of reduced objective functions \(i_{\delta}(\kappa) = \inf_{\qnl_\delta\in \Qd(f)}\Id(\knl;\qnl_\delta)\)
      \(\Gamma\)-converges, as \(\delta\searrow 0\), towards \(i(\kappa) =  \inf_{q\in \Q(f)} \Idloc(\kappa;q)\) on the set
      \(\kadm_{\bar{\delta}}\subset L^\infty(\O_{\bar{\delta}})\), where \(\bar{\delta}\) is a fixed positive number and
      \(L^\infty(\O_{\bar{\delta}})\) is considered with its weak\(^*\) topology.
\end{theorem}

We would like to emphasize that dual formulations of nonlocal problems do not appear very often in the literature, and from this perspective this proof has its own merit. In this case, nonlocal-to-local convergence is obtained in an straightforward way primarily utilizing the stability of the dual variational principle (Theorem~\ref{thm:mixed_exist}). An alternative, direct proof, is also possible. The estimation of \(\limsup_{\delta\searrow 0} p^*_{\delta}\) may be obtained following the lines of~\cite[Proposition~4.4]{evgrafov2019non}  (this reference deals with geometric averaging of conductivities, it is however clear from the proof that this is not a limiting factor). The estimation of \(\liminf_{\delta\searrow 0} p^*_{\delta}\) may be shown  following the ideas of~\cite[Proposition~4.1]{evgrafov2019non} and utilizing the so-called ``generalized Ponce's inequality'', see~\cite{munoz2021generalized,JulioMunoz2021CtGP}.

\bibliographystyle{alpha}
\bibliography{dual_ponce}

\end{document}